\documentclass[12pt, reqno]{amsart}

\usepackage{amsrefs}

\usepackage{amsmath, amsthm, amssymb, stmaryrd}
\usepackage{hyperref}
\usepackage{enumerate}
\usepackage{url}
\usepackage{bbold}
\usepackage{fancyvrb}
\usepackage{xcolor}
\usepackage{comment}
\usepackage{dsfont}

\usepackage{fancyvrb}

\makeatletter
\@namedef{subjclassname@2020}{%
\textup{2020} Mathematics Subject Classification}
\makeatother

\newtheorem{thm}{Theorem}[section]
\newtheorem{lemma}[thm]{Lemma}
\newtheorem{prop}[thm]{Proposition}

\newtheorem{conj}[thm]{Conjecture}

\newtheorem {lem} [thm]    {Lemma}

\newtheorem {prp}[thm]  {Proposition}

\theoremstyle{definition}

\newcounter{AbcT}

\theoremstyle{definition}

\theoremstyle{remark}
\newtheorem{remark}[thm]{Remark}

\numberwithin{equation}{section}

\newcommand{\mmod}[1]{\,\,\text{mod}\,\,#1}

\def\alp{{\alpha}} 
\def\bet{{\beta}}  
 
\def\del{{\delta}} 

\def\tet{{\theta}}  
\def\kap{{\kappa}}

\def\eps{\varepsilon}

\def\le{\leqslant}
\def\leq{\leqslant} \def\ge{\geqslant}
\def\geq{\geqslant}
\def\epsilon{\varepsilon}

\def\d{{\,{\rm d}}}

\def \bC {\mathbb C}

\def \bN {\mathbb N}

\def \bQ {\mathbb Q}
\def \bR {\mathbb R}
\def \bZ {\mathbb Z}

\def \bZ {\mathbb Z}

\def \R {\mathbb R}

\def \Z {\mathbb Z}

\def \cJ {\mathcal J}

\def \cN {\mathcal N}

\def \cX {\mathcal X}

\def \Im {\mathrm{Im}}

\def \dim {\mathrm{dim}}

\def \supp {\mathrm{supp}}
\def \dimH {\mathrm{dim}_{\mathrm{H}}}
\def \Lip {{\mathrm{Lip}}}

\newcommand{\wt}{\widetilde}

\newcommand{\one}{{\mathds{1}}}

\DeclareMathOperator{\dist}{dist}

\renewcommand{\Im}{\operatorname{Im}}

\usepackage{xcomment}

\begin{document}
\title[Rationals in Cantor sets]{Counting rationals and diophantine approximation in missing-digit Cantor sets}
\subjclass[2020]{11J83 (primary); 28A78, 28A80, 42A16 (secondary)}
\keywords{diophantine approximation, fractals, Fourier series}
\author{Sam Chow \and P\'eter P. Varj\'u \and Han Yu}

\address{SC, Mathematics Institute, Zeeman Building, University of Warwick, Coventry CV4 7AL, UK}
\email{Sam.Chow@warwick.ac.uk}

\address{PV, Centre for Mathematical Sciences, Wilberforce Road, Cambridge CB3 0WA, UK}
\email{pv270@dpmms.cam.ac.uk}

\address{HY, College of Mathematics and Statistics, Center of Mathematics, Chongqing University, Chongqing, 401331, China}
\email{han.yu.2@cqu.edu.cn}

\thanks{SC was supported by EPSRC Fellowship Grant EP/S00226X/2, and by the Swedish Research Council under grant no. 2016-06596.
PV and HY have received funding from the European Research Council (ERC) under the European Union’s Horizon 2020 research and innovation programme (grant agreement No. 803711). PV has also received funding from the Royal Society. HY has also received funding from the University of Cambridge and Corpus Christi College, Cambridge, as well as from the Leverhulme Trust (ECF-2023-186).}

\begin{abstract}
We establish a new upper bound for the number of rationals up to a given height in a missing-digit set, making progress towards a conjecture of Broderick, Fishman, and Reich. 
This enables us to make novel progress towards another conjecture of those authors about the corresponding intrinsic diophantine approximation problem.
Moreover, we make further progress towards conjectures of Bugeaud--Durand and Levesley--Salp--Velani on the distribution of diophantine exponents in missing-digit sets.

A key tool in our study is Fourier $\ell^1$ dimension introduced by the last named
author in [H. Yu, Rational points near self-similar sets, arXiv:2101.05910].
An important technical contribution of the paper is a method to compute this quantity.
\end{abstract}

\maketitle

\section{Introduction}
In this paper, we study the distribution of rational numbers around missing-digit sets.
Let $b \ge 2$ be an integer, and let $D$ be a subset of $\{ 0,1,\dots,b-1\}$ with $|D| \ge 2$. Let $K_{b,D}$ be the set of real numbers in $[0,1]$ having at least one base $b$ expansion with digits only from the set $D$. We call such a set $K_{b,D}$ a \emph{missing-digit set} with base $b$ and digit set $D.$ For example, the set $K_{3,\{0,2\}}$ is the middle-third Cantor set.

Furthermore, let $p = (p_0,\dots,p_{b-1})$ be a probability vector. Let 
$D_1, D_2, \ldots$ be an i.i.d sequence where $D_i$ takes the value $j$ with probability $p_j$. The \emph{missing-digit measure} $\nu_p$ associated to $p$ is the distribution of the random sum 
\[
\sum_{i\geq 1} D_i b^{-i}.
\]
When $\supp(p) := \{j\in\{0,\dots,b-1\}: p_j\neq 0\} = D$, the measure $\nu_p$ is supported on $K_{b,D}$.
The case when $p$ is the uniform probability vector on $D$, that is, it gives weight $\# D^{-1}$ to each element of $D$, is of special importance to us.
In this case, we call the associated measure the \emph{natural measure} on $K_{b,D}$
and denote it by $\nu_{b,D}$.

Our main results are Theorems \ref{CountingThm}, \ref{IntrinsicThm} and \ref{BDtype}. Further results can be found in \S \ref{sec: fourier l1}.

\subsection*{Notation}
We use the Vinogradov notations $\ll,\gg,\asymp$, as well as the Bachmann--Landau notation $O(\cdot)$.
Let $f,g$ be complex-valued functions. We write $f \ll g$ or $f=O(g)$ if $|f|\leq C|g|$ pointwise, for some constant $C>0$. We write $f\asymp g$ if $f\ll g$ and $g\ll f$.

We denote by $\dimH(B)$ the Hausdorff dimension of a Borel set $B$.

\subsection{Counting rational numbers in missing-digit sets}
Let \mbox{$K\subset\mathbb{R}$} be compact. For each integer $T\geq 1,$ we define
\[
\mathcal{N}_K(T)=\#\{x\in K: qx\in\mathbb{Z} \text{ for some } 1\leq q\leq T\}.
\]
The following is a natural extension of a conjecture due to Broderick, Fishman and Reich \cite{BFR2011}.

\begin{conj}
\label{conj: missing digits Serre}
Let $K$ be a missing-digit set, and let 
$\kappa = \dimH(K).$ 
Then, for each $\varepsilon>0$, we have
\[
\cN_K(T) \ll T^{\kappa+\varepsilon}.
\]
\end{conj}

In \cite{BFR2011}, this conjecture was stated only for $K$ being the middle third Cantor set. See also \cite[Conjecture 1 and Proposition 3.2]{RSTW2020}. Following \cite[Proposition 3.3]{RSTW2020} and \cite[Theorem 4.1]{Sch2020}, it can be shown that
\begin{equation}
\label{LowerBound}
T^{\kappa} \log T\ll\cN_K(T) \ll T^{2\kappa}.\end{equation}
To our knowledge, these are the strongest known bounds on this counting function, prior to this work.

One can view Conjecture \ref{conj: missing digits Serre} as an analogue of Serre's \emph{dimension growth conjecture}, which has been proved by Salberger, see \cite{CCDN2020}. This asserts that an integral projective variety $X$ of degree $d \ge 2$ over $\bQ$ has $O_{X,\eps}(T^{\dim(X)+\eps})$ rational points of height at most $T$. 

\bigskip

Our first result is a sharper upper bound for $\mathcal{N}_K(T)$ for a class of missing-digit sets, wherein one digit is missing.

\begin{thm}[Main Theorem 1] \label{CountingThm} Let $K=K_{b,D}$ be such that 
\begin{enumerate}[(i)]
\item $b\geq 5$ and $\#D=b-1$, or
\item $b=4$ and $D = \{0,1,2\}$ or $\{1,2,3\}$. 
\end{enumerate}
Then there exists an effectively-computable constant $\varrho > 0$ such that
\begin{equation} \label{PowerSaving}
\cN_{K}(T) \ll T^{2 \kap - \varrho}.
\end{equation}
\end{thm}

This has implications for \emph{intrinsic diophantine approximation}, as we now explain. Let $K$ be a closed subset of $\bR$. For $\alp \ge 1$, denote by $W_K(\alp)$ the set of $x \in \bR$ for which
\[
|x-p/q| < q^{-\alp}, \qquad p/q \in K
\]
has infinitely many solutions $(p,q) \in \bZ \times \bN$. The set of \emph{intrinsically very well approximable numbers} is
\[
\qquad \mathbf{VWA}_K := \bigcup_{\alp > 1} W_K(\alp).
\]

\begin{conj}
[Broderick, Fishman and Reich \cite{BFR2011}] \label{IntrinsicVWAconjecture} Let $K$ be a missing-digit set with natural measure $\nu$. 
Then
\[
\nu(\mathbf{VWA}_K) = 0.
\]
Equivalently, if $\alp > 1$ then
\[
\nu(W_K(\alp)) = 0.
\]
\end{conj}

Conjecture \ref{conj: missing digits Serre} implies Conjecture \ref{IntrinsicVWAconjecture}, as discussed in \cite{BFR2011}. Weiss \cite{Weiss2001} showed that 
\begin{equation} \label{WeissExponent}
\nu(W_K(\alp)) = 0 \qquad (\alp > 2).
\end{equation}
We are able to improve (\ref{WeissExponent}) for a large class of missing-digit sets.
\begin{thm}[Main Theorem 2] \label{IntrinsicThm} Let $K=K_{b,D}$ be such that $b\geq 5$ and $\#D=b-1$ or $b=4$ and $D=\{1,2,3\}$ or $\{0,1,2\}$. Let $\nu$ be the natural measure on $K.$ Then there exists an effectively-computable $\varrho > 0$ such that
\[
\nu(W_{K}(2 - \varrho)) = 0.
\]
\end{thm}

\begin{remark} What if we choose $K=[0,1]$ and let $\nu$ be the Lebesgue measure? Dirichlet's approximation theorem implies that $W_{[0,1]}(2)$ contains $[0,1]$ and therefore $\nu(W_{[0,1]}(2))=1.$ Thus, we see that the metric theory of intrinsic diophantine approximation for $[0,1]$ is very different to that of a proper missing-digit set, i.e. a missing-digit set which is not the whole interval $[0,1].$
\end{remark}

Some readers may have seen a variant of the problem in which the intrinsic denominator is used for the middle-third Cantor set $K$. A rational number $p/q \in K$ must have preperiodic ternary expansion. If its initial block has length $i_0 \in \bZ_{\ge 0}$ and its period is $\ell \in \bN$, then its \emph{intrinsic denominator} is
\[
q_{\mathrm{int}} = 3^{i_0} (3^\ell - 1). 
\]
Counting rationals in $K$ by intrinsic height is a simpler task, and by a short calculation
\[
\# \{p/q \in K: q_{\mathrm{int}} \le T \} \ll T^{\kap + \eps}.
\]
The metric theory of intrinsic diophantine approximation with this height is complete, owing to the efforts of Fishman--Simmons \cite{FS2014} and Tan--Wang--Wu \cite{TWW}.

\subsection{Extrinsic approximation}

Our next result concerns approximation of elements of missing-digit sets with rationals
that are not necessarily contained in the missing-digit set. 
We consider a natural generalisation of a conjecture by Bugeaud and Durand \cite[Conjecture 1]{BD2016}, which was originally made for the middle-third Cantor set. See also \cite[Conjecture B]{Yu2021}. 

\begin{conj} \label{BDconj} Let $K$ be a missing-digit set of Hausdorff dimension $\kap$. Then, for $\alp \ge 2$, we have
\begin{equation} \label{BDeq}
\dimH(W_{\mathbb{R}}(\alp) \cap K) = 
\max \left \{ \frac2{\alp} + \kap - 1, \frac{\kap}{\alp} \right \}.
\end{equation}
\end{conj}

Taking $\alp \to 2^+$ gives rise to the following natural generalisation of the conjecture \cite[Equation (29)]{LSV2007} of Levesley, Salp and Velani, as stated for the middle-third Cantor set in \cite[Conjecture A]{Yu2021}. Recall that the set of \emph{very well approximable numbers} is
\[
\mathbf{VWA} := \bigcup_{\alp > 2} W_{\mathbb{R}}(\alp).
\]

Notice that the union is taken for exponents greater than $2$ rather than $1$ as in the previous section.
The reason for this is that now we consider approximations by all rationals not by only those
that are contained in $K$, which makes the critical threshold for the exponent larger.
 
\begin{conj} \label{LSVconj}
Let $K$ be a missing-digit set. Then
\begin{equation}
\label{DimConj}
\dimH(\mathbf{VWA} \cap K) = \dimH(K).
\end{equation}
\end{conj}

The last named author made progress towards these conjectures in \cite[Theorem A and Remark A]{Yu2021}, showing that Conjecture \ref{BDconj} holds for missing-digit sets with base exceeding $10^7$, few missing digits and $\alp$ being sufficiently close to $2.$ We are able to extend this result to missing-digit sets that may not have large bases.

\begin{thm}[Main Theorem 3] \label{BDtype} Let $K=K_{b,D}$ with $b \ge 7$ and $\#D=b-1$. 
Let $\kappa=\dimH(K)$.
Then there is an effectively-computable $\varrho>0$ such that
\[
\dimH(W_\R(\alp)\cap K)= \frac{2}{\alp}+\kappa-1
\] holds for $\alp\in [2,2+\varrho)$.
In particular, this implies \eqref{DimConj}.
\end{thm}

\subsection{Organisation of the paper}
In \S\ref{sec: fourier l1}, we recall the notion of Fourier $\ell^1$ dimension from \cite{Yu2021}, and explain
its role in the proof of the main results.
In \S\ref{ADreg}, we prove results that reduce the main theorems to
estimating the Fourier $\ell^1$
dimensions of missing-digit measures.
We give upper and lower bounds on the Fourier $\ell^1$ dimension of missing
digit measures in \S\ref{sec: l1 computation}.
This leads to an algorithm to compute Fourier $\ell^1$ dimension to arbitrary
precision.
In \S\ref{sc:analytic}, we use the bounds in \S\ref{sec: l1 computation}
to prove the main results for large $b$.
In \S\ref{sc:numeric}, we carry out numerical calculations for the bounds
in \S\ref{sec: l1 computation} to prove the main results for the remaining
small values of $b$.

\subsection*{Acknowledgements}
First and foremost, we thank Demi Allen for many productive discussions, which greatly contributed to our project. We also thank her for helpful comments on an earlier draft.

We thank Dmitry Kleinbock for a helpful suggestion. We thank Fredrik Johansson for helpful advice on numerical computations involving interval arithmetic.
Finally, we thank an anonymous referee for a careful reading of the paper.

\subsection*{Rights}

For the purpose of open access, the authors have applied a Creative Commons Attribution (CC-BY) licence to any Author Accepted Manuscript version arising from this submission.

\section{Fourier \texorpdfstring{$\ell^1$}{l1} dimension and applications} \label{sec: fourier l1}

Theorems \ref{CountingThm}, \ref{IntrinsicThm}, \ref{BDtype} are all proved via a Fourier 
analytic method. We now introduce the notion of Fourier $\ell^1$ dimension. We write
\begin{equation}\label{eq:e}
	e(x) = e^{2 \pi i x}
\end{equation}
throughout. Let $\nu$ be a Borel probability measure supported on $[0,1]$.  For $\xi \in \bZ$, the $\xi^{\mathrm{th}}$ \emph{Fourier coefficient} of $\nu$ is
\[
\hat \nu(\xi) = \int_0^1 e(-\xi x) \d \nu(x),
\]
see for instance \cite{Gra2014}. For $t \ge 1$, the \emph{Fourier $\ell^t$ dimension} of $\nu$ is
\[
\hat\kap_t(\nu) = \sup \left \{ s \ge 0: \sum_{\xi = 0}^Q |\hat \nu(\xi)|^t \ll_{\nu,s} Q^{1-s} \right \}.
\]
This concept was introduced in \cite{Yu2021}. Note that $\hat \kap_t(\nu)$ is increasing as a function of $t$. 

Our next result gives a sufficient condition in terms of the Fourier $\ell^1$ dimension of the natural measure for the conclusions of Theorems~\ref{CountingThm} and
\ref{IntrinsicThm} to hold.
This result holds in greater generality, and we introduce the relevant terminology.
Let $\nu$ be a Borel measure on $\bR$. For $s > 0$, the measure $\nu$ is \emph{$s$-regular}, or \emph{AD-regular with exponent $s$}, if there is a constant $C > 1$ such that if $x\in \supp(\nu)$ and $r>0$ is sufficiently small then
\[
C^{-1} r^s\leq \nu(B_r(x))\leq C r^s,
\] 
where $B_r(x)$ is the interval of radius $r$ and centre $x$.
For an AD-regular measure $\nu$, the exponent is easily seen to be
\[
s = \dimH(\supp(\nu)),
\]
and we define
\[
\dimH(\nu)
= \dimH(\supp(\nu)).
\]
Moreover, for any AD-regular measure $\nu$, we have the following discrete analogue of an estimate from \cite[Subsection 3.8]{Mattila}:
\[
\hat\kap_2(\nu) = \dimH(\nu).
\]
Missing-digit measures are AD-regular, as discussed in \cite[Section 3]{Yu2021}.

\begin{thm}\label{thm: l1 to main}
Let $\kap \in (0,1)$, let $\nu$ be a $\kap$-regular Borel probability measure such that $\hat\kappa_1(\nu)>1/2$, 
and let $K=\supp(\nu)$.
Then the following hold.
\begin{enumerate}[(a)] 
\item There exists an effectively-computable $\varrho > 0$ such that
\[
\cN_{K}(T) \ll T^{2 \kap - \varrho}
\]
\item There exists an effectively-computable $\varrho > 0$ such that 
\[
\nu(W_{K}(2 - \varrho)) = 0.
\]
\end{enumerate}
\end{thm}

\begin{remark}
Since Fourier $\ell^1$ dimension is invariant under affine translations,
and our proofs of Theorems \ref{CountingThm} and \ref{IntrinsicThm}
rely only on Theorem~\ref{thm: l1 to main} and lower bounds on $\hat\kappa_1(\nu)$,
Theorems  \ref{CountingThm} and \ref{IntrinsicThm} also hold for any translates
of the missing-digit sets for which they are stated.
\end{remark}

The last named author gave a sufficient condition, in terms of the Fourier $\ell^1$ dimension of the natural measure, for the conclusion of Theorem \ref{BDtype} to hold.

\begin{thm}[Yu \cite{Yu2021}*{Theorem 9.5}]
	\label{thm: l1 to BD}
	Let $K$ be a missing-digit set, and suppose its natural measure $\nu$ satisfies 
	\[
	\hat\kappa_1(\nu)\dimH(\nu)>1/2.
	\]
	Then there exists $\varrho > 0$ such that
	\[\dimH(W_{\mathbb{R}}(\alp) \cap K) = \frac2{\alp} + \kap - 1
	\]
	holds for $\alp\in [2,2+\varrho)$.
\end{thm}

These two results reduce our main theorems, Theorems \ref{CountingThm}, \ref{IntrinsicThm}
and \ref{BDtype}, to proving suitable estimates for the Fourier $\ell^1$ dimension
of missing-digit measures.
Specifically, we will prove the following results, which will complete
the proofs of the main theorems.

\begin{prop}\label{pr:kap1>12}
Let $b\ge 5$ and $D\subset\{0,1,\ldots,b-1\}$ with $\# D=b-1$, or let
$b=4$ and $D=\{0,1,2\}$ or $\{1,2,3\}$.
Then $\hat\kappa_1(\nu_{b,D})>1/2$.
\end{prop}

\begin{prop}\label{pr:kap1kap>12}
Let $b\ge 7$ and $D\subset\{0,1,\ldots,b-1\}$ with $\# D=b-1$.
Then
\[
\hat\kappa_1(\nu_{b,D})\dimH(\nu_{b,D})>1/2.
\]
\end{prop}

In \S\ref{sec: l1 computation}, we give sequences of easily computable
lower and upper bounds for the Fourier $\ell^1$ dimension of missing-digit measures.
We also show that the lower and upper bounds converge to the same limit, which yields
an algorithm that allows us to compute $\hat\kappa_1(\nu)$ to arbitrary precision.
See Theorem \ref{th:bounds} for more details.
This result will also be used to prove
Propositions \ref{pr:kap1>12} and \ref{pr:kap1kap>12}, and we will also
deduce the following.

\begin{thm} \label{thm: l1 bound for special arrangement}
Let $b \ge 3$ be an integer, and let 
\[
D =  \{kd+a:k=0,\ldots,l-1\},
\]
for some integers
$a \ge 0$, $d \ge 1$ and $l \ge 2$.
Assume $D$ is a proper subset of $\{0,\ldots,b-1\}$. 
Let $\nu$ be the natural missing-digit measure on $K_{b,D}$. 
Then 
\[
\dimH(\nu)
\geq\hat\kappa_1(\nu)
\geq \dimH(\nu)
- \frac{\log(4+ \log 2l)}
{\log b}.
\]
\end{thm}

\section{Rationals in the support of AD-regular measures}\label{ADreg}

In this section, we establish Theorem \ref{thm: l1 to main}. 

\subsection{From counting rationals to intrinsic approximation}

The next lemma reduces part (b) of Theorem \ref{thm: l1 to main} to proving
part (a).

\begin{lemma} Let $\nu$ be a $\kap$-regular Borel probability measure supported on $K \subset \bR$, and let $\alp,\bet$ be positive constants such that $\alp > \bet/\kap$. If $\cN_K(T) \ll_{K,\bet} T^\bet$ then $\nu(W_K(\alp)) = 0$.
\end{lemma}

\begin{proof} Observe that
\[
W_K(\alp) = \limsup_{n \to \infty} A_n,
\]
where
\[
A_n = \bigcup_{\substack{
p/q \in K \\ e^n < q \le e^{n+1}}} 
\left( \frac p q \: - q^{-\alp}, \frac p q + q^{-\alp} \right) \qquad (n \in \bN).
\]
If $\cN_K(T) \ll T^\bet$, then, by the AD-regularity of $\nu$, we have
\[
\nu(A_n) \ll (e^{\bet-\kap \alp})^n,
\]
so
\[
\sum_{n=1}^\infty \nu(A_n) < \infty.
\]
By the first Borel--Cantelli lemma, we then have $\nu(W_K(\alp)) = 0$.
\end{proof}

\subsection{From Fourier \texorpdfstring{$\ell^1$}{l1} dimension to counting rationals}

Throughout this subsection, let $\nu$ be a $\kap$-regular Borel probability measure supported on a compact set $K \subset \mathbb{R}$, where $\kap < 1$.
We prove the following result, which completes the proof of Theorem \ref{thm: l1 to main}.

\begin{prp} \label{GeneralPrp} Suppose $\hat\kap_1(\nu) > 1/2$.
Then we have 
\[
\cN_K(T)\ll T^{2\kappa-\rho}
\]
for some $\varrho > 0$.
\end{prp}

We begin with two lemmata.
For $q \in \bN$ and $\del \in (0,1/2)$, set
\[
A(q, \del) = \{ x \in \bR: \| q x \| \le \del \},
\]
where $\|\cdot\|$ denotes distance to the nearest integer.

\begin{lemma} Let $v$ be a number in the range $\hat\kap_1(\nu) > v > 1/2$. Then
\begin{equation} \label{average}
\sum_{Q \le q \le 2Q} \nu(A(q,\del)) \ll \del Q,
\end{equation}
for $Q \in \bN$ and $\del \ge Q^{-v/(1-v)}$.
\end{lemma}

\begin{proof}
We may assume that $Q$ is large. Put $\eps = \hat\kap_1(\nu) - v > 0$. 
For an upper bound, we may replace the indicator function of $A(q,\delta)$ by
\[
x \mapsto \sum_{a \in \bZ} \phi\left(
\frac{qx - a}{\del} \right).
\]
Here $\phi: [0,\infty)$ is a Schwartz function such that $\phi \ge 1$ on $[-1,1]$ and $\hat \phi$ is supported on $[-1,1]$, a construction of which is given in \cite[Example 3.2]{Mattila}. Applying Parseval's formula --- see \cite[Theorem 4.1]{Yu2021} for details --- gives
\[
\nu(A(q,\del)) \ll \del 
\sum_{\substack{|\xi| \le 2q/\del \\ \xi \equiv 0 \mmod q}}
|\hat{\nu}(\xi)|,
\]
and so
\begin{align*}
\sum_{Q \le q \le 2Q} \nu(A(q,\del)) 
&\ll \del \Big( Q|\hat\nu(0)|+
 \sum_{0<|\xi| \le 4Q/\del} \tau(\xi) |\hat{\nu}(\xi)|\Big)\\
 &\ll \del \Big( Q+ (Q/\del)^{\eps/2} \sum_{|\xi| \le 4Q/\del} |\hat{\nu}(\xi)| \Big)\\
& \ll \del ( Q+ (Q/\del)^{\eps + 1 - \hat\kap_1(\nu)})\\
&\ll \delta (Q+(Q/\delta)^{1-v}).
\end{align*}
Here $\tau(\xi)$ denotes the number of divisors of $\xi$.
We thus have \eqref{average} as long as $\del \ge Q^{-u}$, where
\[
u = \frac{1}{1-v} - 1
= \frac{v}{1-v}.
\]
\end{proof}

It is convenient to define
\[
\cN_{K}^*(T) = 
\# \{ (p,q) \in \bZ^2: T/2 \le q \le T, p/q \in K\}
\]
for $T \ge 3$, this being a dyadic variant of $\cN_{K}(T)$. 

\begin{lemma}\label{lma: intrinsic counting} Let $v$ be a number in the range $\hat\kap_1(\nu) > v > 1/2$. Then \[
\cN_{K}^*(2Q) \ll_v Q^E,
\]
where
\[
E = 2\kap - \frac{(1-\kap)(2v - 1)}{1-v}.
\]
\end{lemma}

\begin{proof}
We may assume that $Q$ is large in terms of $v$. Put
\[
\del = Q^{-v/(1-v)} < 1/(8Q).
\]
If $a_1/q_1 \ne a_2/q_2$ for some $a_1,a_2\in \Z$ and $Q \le q_1,q_2 \le 2Q$,
then
\[
|a_1/q_1-a_2/q_2|\ge 1/(q_1q_2)\ge1/(4Q^2)>2\delta/Q
\]
and hence the intervals
\[
\left(
\frac{a_1 - \del}{q_1},
\frac{a_1 + \del}{q_1}
\right),
\left(
\frac{a_2 - \del}{q_2},
\frac{a_2 + \del}{q_2}
\right)
\]
are disjoint.

By AD-regularity of $\nu$ and \eqref{average}, we now have
\[
\cN_{K}^*(2Q) (\del/Q)^\kap
\ll \sum_{Q \le q \le 2Q} \nu(A(q,\del))
 \ll \del Q,
\]
so
\[
\cN_{K}^*(2Q) \ll \del^{1-\kap} Q^{1+\kap}
= Q^{1+\kap - (1-\kap)v/(1-v)}.
\]
The exponent is
\[
1+\kap - \: 
\frac{(1-\kap)v}{1-v} 
= 2\kap - \: 
\frac{(1-\kap)(2v - 1)}{1-v}.
\]
\end{proof}

\begin{proof}[Proof of Proposition \ref{GeneralPrp}]
We have
\begin{align*}
\cN_{K}(T) &=\cN_{K}^*(T)+\cN_{K}^*(T/2)+\cN_{K}^*(T/4)+\ldots\\
&\ll T^E +(T/2)^E+(T/4)^E+\ldots\ll T^E
\end{align*}
by Lemma \ref{lma: intrinsic counting}.
Therefore, the proposition holds for any
\[
\varrho \le \frac{(1-\kap)(2v - 1)}{1-v},
\]
and the right hand side is positive.
\end{proof}

\section{Approximating Fourier \texorpdfstring{$\ell^1$}{l1} dimension}
\label{sec: l1 computation}

In this section, we approximate the Fourier $\ell^1$ dimension of missing-digit
measures by quantities that are easy to evaluate numerically.
This yields an algorithm to approximate $\hat\kappa_1(\nu)$ to arbitrary precision, and
it will also be the basis of all of our results about estimating the Fourier $\ell^1$
dimension of specific measures.

We begin by introducing some notation.
Let $b \ge 2$ be an integer, and let $p=(p_0,\ldots,p_{b-1})$ be a probability vector.
We allow some of the coordinates to vanish, but we assume that
$p\neq(b^{-1},\ldots,b^{-1})$.
Let $\nu$ be the associated missing-digit measure, that is the
distribution of the random variable
\[
\sum_{j=1}^{\infty} \xi_j b^{-j},
\]
where $\xi_j$ is a sequence of independent random variables
with distribution $p$.

Writing $\mu_p$ for the measure on $\{0,\ldots, b-1\}$ corresponding to $p$,
we have
\[
\hat\nu(x)=\prod_{j=1}^{\infty}\hat\mu_p(b^{-j}x),
\]
since the Fourier transform of the distribution of sums of independent random
variable is the product of the Fourier transform of the distributions of the terms.

Let $g:\R/\Z\to [0,1]$ be defined by
\begin{equation}\label{eq:g}
g(x)=|\hat\mu_p(x)|=\left|\sum_{j=0}^{b-1}p_j e(jx)\right|.
\end{equation}
Then
\begin{equation} \label{AbsExact}
|\hat \nu(x)|=\prod_{j=1}^{\infty} g(b^{-j}x).
\end{equation}

The purpose of this section is to prove the following two results.
The first result establishes that the Fourier $\ell^1$ dimension of the
self-similar measures we consider can be expressed as certain limits.

\begin{prp}\label{pr:limit-exists}
	With the above notation, the Fourier $\ell^1$ dimension of $\nu$ is equal to
	\begin{align*}
		\hat\kappa_1(\nu)=&\lim_{Q\to \infty}\frac{-\log \Big(Q^{-1}\sum_{n=0}^{Q-1} |\hat \nu(n)|\Big)}{\log Q}\\
		=&\lim_{N\to\infty} \frac{-\log\Big(\int_0^1 \prod_{j=0}^{N-1} g(b^j x) \d x\Big)}{\log b^N}.
	\end{align*}
\end{prp}

The second result provides a way to approximate the Fourier $\ell^1$ dimension.
We introduce the notation 
\begin{equation}\label{eq:SL}
S_L(x)=\prod_{j=0}^{L-1} g(b^{j}x).
\end{equation}

\begin{thm}\label{th:bounds}
	With the above notation, for all $L \in \bN$, we have
	\begin{align*}
		&\frac{-\log\Big(\max_{x}b^{-L}\sum_{i=0}^{b^L-1} S_L(x+i/b^L)\Big)}{\log b^L}\\
		&\qquad\qquad\qquad\qquad\le \hat\kappa_1(\nu)\\
		&\qquad\qquad\qquad\qquad\le \frac{-\log\Big(\min_{x}b^{-L}\sum_{i=0}^{b^L-1} S_L(x+i/b^L)\Big)}{\log b^L}.
	\end{align*}
	Moreover, the above upper and lower bounds converge to the same limit as $L\to\infty$.
\end{thm}

\noindent Proposition \ref{pr:convergence} below provides an explicit estimate
for the difference between the upper and lower bounds.

The rest of this section is devoted to the proof of the above two results.
First we give a short direct proof of the lower bound in Theorem 4.2
in Section \ref{sc:lower}, and then
we present the full proof of both results starting in Section \ref{sc:Riemann}.
The full proofs will not rely on Section \ref{sc:lower} in any way, and the reader
may skip that section.
Our reasons for presenting a separate proof of the lower bound are twofold.
First, it allows us to present the main ideas without the technicalities of the full proofs.
Second, Propositions \ref{pr:kap1>12} and \ref{pr:kap1kap>12} and therefore the
main results of the paper only rely on this lower bound.

\subsection{Lower bound}\label{sc:lower}

The purpose of this section is to give a short direct proof of
the lower bound
\[
	\frac{-\log\Big(\max_{x}b^{-L}\sum_{i=0}^{b^L-1} S_L(x+i/b^L)\Big)}{\log b^L}
	\le \hat\kappa_1(\nu)
\]
claimed in Theorem \ref{th:bounds}.

We first show that we can assume without loss of generality that $L=1$.
Indeed, if $L>1$, we can replace $b$ by $b^L$ and the probability vector
$p$ by the distribution of 
\[
\sum_{j=0}^{L-1}\xi_j b^j,
\]
where the $\xi_j$ are independent random variables distributed according to $p$.
This leaves the missing-digit measure $\nu$ unchanged, and it replaces $S_L$
by $S_1=g$.

From now on we assume $L=1$.
By the definition of $\hat\kappa_1$, it is enough to prove the following.

\begin{lemma}\label{lm:lower}
Let
\[
s=\frac{-\log\Big(\max_{x}b^{-1}\sum_{i=0}^{b-1} g(x+i/b)\Big)}{\log b}.
\]
Then
\[
\sum_{\xi=0}^{Q}|\hat\nu(\xi)|\ll Q^{1-s}
\]
for all $Q\in\Z_{>0}$.
\end{lemma}

\begin{proof}
We observe that from the claim for $Q=b^N-1$, we can deduce
it for all $Q\in[b^{N-1},b^N)$, so it is enough to prove it
for $Q=b^N-1$ for $N\in\Z_{>0}$.

We prove the following inequality by induction on $N$.
For $N\in\Z_{\ge0}$ and $y\in\R$, we have
\[
\sum_{\xi=0}^{b^N-1}|\hat\nu(y+\xi)|
\le \Big(\max_x \sum_{\xi=0}^{b-1} g(x+\xi/b)\Big)^N
=b^{N(1-s)}.
\]
From this the lemma follows.

For $N=0$, the claim follows from the fact that $|\hat\nu(y)|\le 1$ for all $y$.
We assume that $N>0$, and the claim holds for $N-1$.
By \eqref{AbsExact}, we have
\[
|\hat\nu(\xi)|=g(b^{-1}\xi)|\hat\nu(b^{-1}\xi)|.
\]
Writing $\xi=\xi_1+b\xi_2$,
we have
\[
\sum_{\xi=0}^{b^N-1}|\hat\nu(y+\xi)|
\le\sum_{\xi_1=0}^{b-1}\sum_{\xi_2=0}^{b^{N-1}-1}g(b^{-1}y+b^{-1}\xi_1+\xi_2)
|\hat\nu(b^{-1}y+b^{-1}\xi_1+\xi_2)|.
\]
Using that $g$ is 
$\Z$-periodic, we can write
\[
\sum_{\xi=0}^{b^N-1}|\hat\nu(y+\xi)|
\le\sum_{\xi_1=0}^{b-1}g(b^{-1}y+b^{-1}\xi_1)
\sum_{\xi_2=0}^{b^{N-1}-1}|\hat\nu(b^{-1}y+b^{-1}\xi_1+\xi_2)|.
\]
By the induction hypothesis, we now have
\begin{align*}
\sum_{\xi=0}^{b^N-1}|\hat\nu(y+\xi)|
&\le\sum_{\xi_1=0}^{b-1}g(b^{-1}y+b^{-1}\xi_1)
\Big(\max_x \sum_{\xi=0}^{b-1} g(x+\xi/b)\Big)^{N-1}\\
&\le \Big(\max_x \sum_{\xi=0}^{b-1} g(x+\xi/b)\Big)^{N}.
\end{align*}
This completes the proof.
\end{proof}

Observe that the same proof cannot be used to prove the upper bound in
Theorem \ref{th:bounds}.
Indeed, the initial step of the induction would require the bound $|\hat\nu(y)|\ge 1$,
which is not true for all $y$.
The proof of the upper bound and the fact that the two bounds converge to the same limit
as $L\to \infty$ is more subtle and it is affected by the location of the zeros of the
function $g$.

\subsection{Comparing two Riemann sums}\label{sc:Riemann}

In what follows we prove Proposition \ref{pr:limit-exists} and Theorem \ref{th:bounds}
without using anything from Subsection~\ref{sc:lower}.
We begin with some preparations.
The purpose of this section is to prove a result
estimating the ratio of two Riemann sums of the function $S_L$.
This will be key for the proofs of both Proposition~\ref{pr:limit-exists}
and Theorem \ref{th:bounds}.

Before stating this, we record some simple properties of the function $g$, which we will
use repeatedly.
We denote by $Z$ the multiset of zeros of $g$.
We define the functions $f$ and $P$ implicitly by
\[
g(x)=f(x)\prod_{z\in Z}\dist(x,z)=f(x)P(x).
\]
Here $\dist(\cdot)$ refers to the natural distance in $\R/\Z$. We regard $b$ and $p$ as fixed, so the constants $c_0$, $C_0$, $C_1$ and $C_2$ below are allowed to depend on them.

\begin{lem}\label{lm:f}
The following hold:
\begin{enumerate}
\item $g(0)=1$,
\item $|Z|\le b-1$,
\item there is a constant $C_0>0$ such that $g(x) \ge \exp(-C_0 x^2)$ for all $x\in[-1/(2b),1/(2b)]$,
\item there is a constant $c_0 \in (0,1/2]$ such that
for all $x\in\R/\Z$, there are at most $b-2$ values of $a \in \{0,\ldots,b-1\}$
such that $\dist(x+a/b,Z)\le c_0/b$,
\item $f'$ exists and is continuous apart from finitely many singularities, where
$f'$ has a jump,
\item there is a constant $C_1$ such that $|f'(x)|/f(x)\le C_1$ for all $x\in\R/\Z$,
\item there is a constant $C_2$ such that $f(x) \le  C_2$ for all 
$x \in \bR / \bZ$.
\end{enumerate}
\end{lem}

\begin{proof}
Item (1) holds because $p$ is a probability vector. Item (2) holds because, $g(x)=|Q(e(x))|$ for a polynomial $Q$ of degree at most $b-1$.
	
For Item (3), since $g(0) = 1$, we may by symmetry assume that $0 < x \le 1/(2b)$. Observe that $\Im(e(jx)) > 0$ for all $j \in \{ 1,\ldots, b-1 \}$, hence
\[
\Im \left(\sum_{j=0}^{b-1}p_j e(jx) \right)>0
\]
unless $p_0=1$.
In either case, we have $|g(x)|>0$ for $|x| \le 1/(2b)$. The existence of $C_0$ follows from the smoothness of $g$ at $0$, continuity, and compactness. 
	
We turn to Item (4).
We first show that $g$ cannot vanish along an arithmetic progression
of step $1/b$ and length $b-1$ in $\R/\Z$.
Suppose for a contradiction that $g$ vanishes at $x+a/b$ for some $x\in \R/\Z$ and all 
$a \in \{ 1,\ldots,b-1 \}$.
Then 
\[
g(x)=|Q(e(x))|,
\]
where 
\[
Q(z) = \sum_{j=0}^{b-1}
p_j z^j
\]
is a polynomial of degree at most $b-1$ that vanishes at $e(x+a/b)$ for $a=1,\ldots,b-1$.
It follows that
\begin{align*}
Q(z) &= \wt c \prod_{a=1}^{b-1}(z-e(x+a/b))
 = \wt c e(x)^{b-1}\prod_{a=1}^{b-1}(e(-x)z-e(a/b))\\
 &=\wt c e(x)^{b-1}\sum_{j=0}^{b-1}(e(-x)z)^j,
\end{align*}
for some $\wt c\in\bC$.
	Since $p$ is a probability vector, this forces $e(-x)=1$ and $\wt c=1/b$.
	However, we explicitly excluded $p=(1/b,\ldots,1/b)$, so this is not possible.
	This contradiction proves our claim.
	The existence of $c_0$ follows again by continuity and compactness.
	
Items (5) and (6) are immediate from the definitions. 

Item (7) holds because $g(x)/P(x)$ is a continuous function on $\R/\Z$.
Indeed, by considering the Taylor series expansion of $g$ around each point $z\in Z$,
it can be seen that $g(x)=O(\dist(x,z)^m)$, where $m$ is the multiplicity of
$z$ in $Z$.
\end{proof}

In what follows, we continue to use $c_0$, $C_0$, $C_1$ and $C_2$ to denote constants
for which Lemma \ref{lm:f} holds.
Now we state the main result of this subsection.

\begin{prp}\label{pr:convergence}
	Let $x_1,x_2\in\R/\Z$ and let
	\[
	\cX_j=\{x_j+a/b^L:a=0,\ldots, b^{L}-1\}
	\]
	for $j=1,2$.
	Then for $L\ge 2^{21} b^8(1+\log C_2/\log b)$, we have
	\[
	\frac{\sum_{x\in\cX_1}S_L(x)}{\sum_{x\in\cX_2}S_L(x)}\le
	4b\exp(b^2 L^{3/4}(C_1+\log L +\log c_0^{-1}+2b)).
	\]
\end{prp}

We introduce some notation needed in the proof of Proposition \ref{pr:convergence}.
Fix $j_0 \in \bN$ such that $L^{3/4}\le b^{j_0}<b L^{3/4}$.
Write $\cJ$
for the collection of intervals of the form
\[
[a/b^{L-1-j_0},(a+1)/b^{L-1-j_0})
\]
for $a=0,\ldots, b^{L-1-j_0}-1$. 
We partition $\cJ$ into two subsets $\cJ^{(1)}$ and $\cJ^{(2)}$ as follows.
Let $J\in\cJ$.
We set $J\in\cJ^{(1)}$ if the number of $n\in\{0,\ldots,L-1\}$ such that there is $x$ with
$\dist(x, J)\le L b^{-L}$ and $b^n x\in Z$ is less than $L^{3/4}/(2b)$.
Otherwise we put $J\in \cJ^{(2)}$.

Let us briefly outline the proof of Proposition \ref{pr:convergence},
which occupies the rest of this section.
If $J\in\cJ^{(1)}$, we show that the values of $S_L$ on $J$ are comparable to each other
within a multiplicative error of the rough size specified in the proposition
except for points in small neighbourhoods of the zeros of $S_L$.
We will also show that among the points $x\in\cX_2\cap J$, a substantial proportion
are not too close to a zero of $S_L$ in the above sense.
If $J\in\cJ^{(2)}$, we will show that $S_L$ is very small on $J$ making the
contribution of $\cX_1\cap J$ negligible.

\bigskip

We begin by estimating the contribution of the intervals in $\cJ^{(1)}$.

\begin{lem}\label{lm:J1-1}
	Let $J\in \cJ^{(1)}$, and let $x_1,x_2\in J$.
	Suppose that $\dist(b^{n}x_2,Z)\ge \alpha b^{n-L}$ for some $\alpha\in(0,1/2]$ and for all $n=0,\ldots,L-1$.
	Then 
	\[
	\frac{S_L(x_1)}{S_L(x_2)}\le \exp(b^2L^{3/4}(C_1+\log L+\log \alpha^{-1}+2b)).
	\]
\end{lem}

\begin{proof}
    We write
	\[
	G_L(x)=\prod_{n=0}^{L-1} f(b^n x).
	\]
	We note that
	\[
	\frac{|G_L'(x)|}{G_L(x)} \le \sum_{n=0}^{L-1} b^n\frac{|f'(b^n x)|}{f(b^n x)}
	\le \frac{b^L-1}{b-1} C_1.
	\]
	This means that
	\begin{align*}
		|\log(G_L(x_2))-\log(G_L(x_1))|
		&\le \Big|\int_{x_1}^{x_2}\frac{G_L'(x)}{G_L} \d x\Big|
		\le |x_1-x_2| \frac{b^L-1}{b-1} C_1\\
		&\le b^{j_0+1} C_1
		\le b^2 L^{3/4} C_1.
	\end{align*}
	Thus
	\[
	G_L(x_1)\le \exp(b^2 L^{3/4} C_1) G_L(x_2).
	\]
	
	Now we consider the contribution of the factor $P(b^nx)$ in $S_L$.
	We recall that
	\[
	\frac{P(b^n x_1)}{P(b^n x_2)}
	=\prod_{z\in Z}\frac{\dist(b^n x_1,z)}{\dist(b^n x_2,z)}.
	\]
	By assumption, $\dist(b^n x_2,z)\ge \alpha b^{n-L}$.
	In addition,
	\[
	\dist(b^n x_1,z)\le b^n \dist(x_1,x_2)+\dist(b^n x_2,z)
	\le b^{n-L+2} L^{3/4}+\dist(b^n x_2,z).
	\]
	Therefore,
	\[
	\frac{\dist(b^n x_1,z)}{\dist(b^n x_2,z)}\le 1+\alpha^{-1} b^2 L^{3/4}.
	\]
	So we get
	\[
	\frac{P(b^n x_1)}{P(b^n x_2)}
	\le(1+\alpha^{-1} b^2 L^{3/4})^{b}.
	\]
	
	We improve on this for those $n\in\{0,\ldots, L-1\}$
	such that there is no $x$ with $\dist(x,J)\le L b^{-L}$
	and $b^nx\in Z$.
	If $n$ has this property, then
	$\dist(b^nx_2,z)\ge L b^{n-L}$ for all $z\in Z$, and estimating as above,
	we get
	\[
	\frac{P(b^n x_1)}{P(b^n x_2)}
	\le\prod_{z\in Z}\frac{\dist(b^n x_1,z)}{\dist(b^n x_2,z)}
	\le (1+L^{-1}b^2L^{3/4})^{b}
	\le\exp(b^3L^{-1/4}). 
	\]
	
Putting the two estimates together, and using the defining property of $\cJ^{(1)}$,
gives 
\begin{align*}
\frac{\prod_{n=0}^{L-1}P(b^n x_1)}{\prod_{n=0}^{L-1}P(b^n x_2)}\le&\exp(b^3L^{-1/4})^L(1+\alpha^{-1} b^2 L^{3/4})^{bL^{3/4}}\\
\le& \exp(b L^{3/4}(\log L+\log\alpha^{-1}+2b)).
\end{align*}
	
	Combining the estimates for $G_L$ and $\prod(P(b^nx))$, we get the claim.
\end{proof}

\begin{lem}\label{lm:J1-2}
Let $J\in\cJ^{(1)}$ and let $x_0\in J$. Let 
\[
\cX=\{x_0+a/b^L:a=0,\ldots,b^L-1\}\cap J.
\]
Then
\[
|\{x\in\cX:\text{ $\dist(b^nx,Z)\ge c_0 b^{n-L}$ for all $0\le n\le L-1$}\}|\ge b^{j_0}/2.
\]
\end{lem}

\begin{proof}
	Fix some $n\in\{0,\ldots, L-1\}$.
	Let $x_1,x_2\in\cX$ and $z\in Z$.
	If
	\[
	\dist(b^n x_1,z),\dist(b^n x_2,z)< c_0 b^{n-L},
	\]
	then $\dist(b^n x_1,b^n x_2)< 2c_0 b^{n-L}$.
	Since $c_0\le 1/2$, this is possible only if $b^n x_1=b^nx_2$.
	This means that
	\begin{align*}
		|\{x\in\cX:\text{ $\dist(b^nx,z)< c_0 b^{n-L}$}\}|
		&\le \max\Big(\frac{b^{-(L-1-j_0)}}{b^{-n}},1\Big)\\
		&=\max(b^{n-L+1+j_0},1),
	\end{align*}
since $b^{-(L-1-j_0)}$ is the length of the interval $J$ and $b^{-n}$
is the spacing between the points in question.
	
	We sum this up for all $z\in Z$ and $n=0,\ldots, L-1$ taking into account that
	$|Z|\le b-1$ and that, for $n=L-1$, the number of $z\in Z$ such that there is
	$x\in \cX$ with $\dist(b^{L-1}x,z)\le c_0 b^{-1}$ is at most $b-2$.
	(This is by Item (4) in Lemma \ref{lm:f}.)
	We also use that $J\in\cJ^{(1)}$, hence the number of $n\in\{0,\ldots, L-1\}$ for which $\{x\in\cX:\text{ $\dist(b^nx,z)< c_0 b^{n-L}$}\}$ is non-empty is at most $L^{3/4}/(2b)\le b^{j_0-1}/2$.
	We then get
	\begin{align*}
		|\{x\in\cX:&\text{ $\dist(b^nx,Z)< c_0 b^{n-L}$ for some $0\le n\le L-1$}\}|\\
		&\le (b-2)b^{j_0}+(b-1)\Big(\sum_{n=L-j_0}^{L-2} b^{n-L+1+j_0}+b^{j_0-1}/2\Big)\\
		&\le b^{j_0+1}-b^{j_0}+b^{j_0}/2
= b^{j_0+1}-b^{j_0}/2.
	\end{align*}
	
	Now we note that $|\cX|=b^{j_0+1}$, and the claim follows.
\end{proof}

\begin{prp}\label{pr:J1}
	Let $J\in\cJ^{(1)}$, and let $x_1,x_2\in J$ lie within distance $b^{-L}$ of the left
	endpoint of $J$.
	Then
	\[
	\frac{\sum_{a=0}^{b^{j_0+1}-1}S_L(x_2+a/b^L)}{\sum_{a=0}^{b^{j_0+1}-1}S_L(x_1+a/b^L)}
	\le 2b\exp(b^2L^{3/4}(C_1+\log L+\log c_0^{-1}+2b)).
	\]
\end{prp}
\begin{proof}
	Combining Lemmata \ref{lm:J1-1} and \ref{lm:J1-2}, it follows that there are at least
	$b^{j_0}/2$ terms in the denominator that are at least
	\[
	\exp(-b^2L^{3/4}(C_1+\log L+\log c_0^{-1}+2b))
	\]
	times the maximum of the terms in the numerator.
	There are $b^{j_0+1}$ terms in total, so the claim follows.
\end{proof}

Now we turn to the intervals in $\cJ^{(2)}$.

\begin{lem}\label{lm:J2}
	Let $J\in \cJ^{(2)}$.
	Suppose $L \ge 2^{21} b^8$.
	Then
	\[
	S_L(x) \le C_2^L b^{-L^{3/2}/(100b^2)}
	\]
	for all $x\in J$.
\end{lem}

\begin{proof}
Let $n\in\{0,\ldots,L-1\}$ be such that there are $y\in\R/\Z$ and $z\in Z$ with
$\dist(J,y)\le Lb^{-L}$ 
and $b^ny=z$. Then
\[
P(b^nx)\le\dist(b^nx,b^ny)\le b^n\dist(x,y)\le b^{n-L}(L+b^2L^{3/4})\le b^{n-L}2Lb^2.
\]
By the definition of $\cJ^{(2)}$, this gives
	\[
	 \prod_{n=0}^{L-1}P(b^nx)
	\le \prod_{n=L-\lceil L^{3/4}/(2b)\rceil}^{L-1}
	b^{n-L}\cdot2Lb^2
	\le b^{-L^{3/2}/8b^2}(2Lb^2)^{L^{3/4}}.
	\]
	Using $\log(2Lb^2)/\log b\le 4 L^{1/2}$, which is true for all $b\ge 2$, we get
	\[
	\prod_{n=0}^{L-1}P(b^nx) \le b^{-L^{3/2}/8b^2+4L^{5/4}}
	\le b^{-L^{3/2}/(100b^2)},
	\]
	which implies the claim by item (7) in Lemma \ref{lm:f}.
\end{proof}

\begin{proof}[Proof of Proposition \ref{pr:convergence}]
	By Proposition \ref{pr:J1}, we have
	\begin{align*}
		\frac{\sum_{J\in \cJ^{(1)}}\sum_{x\in\cX_1\cap J} S_L(x)}
		{\sum_{x\in\cX_2} S_L(x)}
		&\le
		\frac{\sum_{J\in \cJ^{(1)}}\sum_{x\in\cX_1\cap J} S_L(x)}
		{\sum_{J\in \cJ^{(1)}}\sum_{x\in\cX_2\cap J} S_L(x)}\\
		&\le 2b\exp(b^2 L^{3/4}(C_1+\log L +\log c_0^{-1}+2b)).
	\end{align*}
	
	On the other hand, by Lemma \ref{lm:J2}, we have
	\[
	\sum_{J\in \cJ^{(2)}}\sum_{x\in\cX_1\cap J} S_L(x)
	\le C_2^Lb^{L-L^{3/2}/(100b^2)}\le 1,
	\]
	as $L\ge 100^2b^4(1+\log C_2/\log b)$.
	In addition, by applying Proposition \ref{pr:J1} with $x_1 = 0$,
	we get
	\begin{align*}
		\sum_{x\in\cX_2} S_L(x) &\ge \sum_{x\in\cX_2\cap[0,b^{-L-1-j_0}]} S_L(x)\\
		&\ge 1\cdot \exp(-b^{2}L^{3/4}(C_1+\log L+\log c_0^{-1}+2b))/(2b),
	\end{align*}
	since $S_L(0)=1$.
	This gives us
	\[
	\frac{\sum_{J\in \cJ^{(2)}}\sum_{x\in\cX_1\cap J} S_L(x)}{\sum_{x\in\cX_2}S_L(x)}
	\le 2b\exp(b^2 L^{3/4}(C_1+\log L +\log c_0^{-1}+2b)).
	\]
	
	Combining our estimates for the contributions of $\cJ^{(1)}$ and $\cJ^{(2)}$, we get the claim.
\end{proof}

\subsection{Completing the proofs}

The next result establishes Theorem~\ref{th:bounds} assuming Proposition \ref{pr:limit-exists} holds. We will prove Proposition \ref{pr:limit-exists} afterwards.

\begin{lem}
	The limit
	\[
	\lim_{N\to\infty} \frac{-\log\Big(\int_0^1 \prod_{j=0}^{N-1} g(b^j x) \d x\Big)}{\log b^N}
	\]
	exists, and for all $L\in \bN$, it is between 
	\[
	\frac{-\log\Big(\max_{x}b^{-L}\sum_{i=0}^{b^L-1} S_L(x+i/b^L)\Big)}{\log b^L}
	\]
	and
	\[
	\frac{-\log\Big(\min_{x}b^{-L}\sum_{i=0}^{b^L-1} S_L(x+i/b^L)\Big)}{\log b^L}.
	\]
\end{lem}

\begin{proof}
	Fix some value of $L$.
	Notice that 
	\[
	\prod_{j=1}^{N-1} S_L(b^{jL}x) 
	\]
	is $b^{-L}$-periodic.
	This means that
	\begin{align*}
		\int_0^1\prod_{j=0}^{N-1} S_L(b^{jL}x) \d x
		=&\int_0^{b^{-L}}\Big(\sum_{i=0}^{b^L-1} S_L(x+i/b^L)\Big)\prod_{j=1}^{N-1} S_L(b^{jL}x) \d x.
	\end{align*}
	As
	\[
	\int_0^{b^{-L}}\prod_{j=1}^{N-1} S_L(b^{jL}x) \d x=b^{-L}\int_0^1\prod_{j=0}^{N-2} S_L(b^{jL}x) \d x,
	\]
	we have
	\begin{align*}
		&\Big(\min_{x}b^{-L}\sum_{i=0}^{b^L-1} S_L(x+i/b^L)\Big)
		\int_0^1\prod_{j=0}^{N-2} S_L(b^{jL}x) \d x\\
		\le&\int_0^1\prod_{j=0}^{N-1} S_L(b^{jL}x) \d x\\
		\le & \Big(\max_{x}b^{-L}\sum_{i=0}^{b^{L}-1} S_L(x+i/b^L)\Big)
		\int_0^1\prod_{j=0}^{N-2} S_L(b^{jL}x) \d x.
	\end{align*}
	By induction, this leads to
	\begin{align*}
		\Big(\min_{x}b^{-L}\sum_{i=0}^{b^L-1} S_L(x+i/b^L)\Big)^N
		&\le \int_0^1\prod_{j=0}^{N-1} S_L(b^{jL}x) \d x \\
		&\le \Big(\max_{x}b^{-L}\sum_{i=0}^{b^L-1} S_L(x+i/b^L)\Big)^N.
	\end{align*}
	
	Observe that, for all $N \in \bN$, we have
	\[
	\int_0^1\prod_{j=0}^{\lfloor N/L\rfloor-1} S_L(b^{jL}x) \d x
	\ge\int_0^1\prod_{j=0}^{N-1} g(b^{j}x) \d x
	\ge\int_0^1\prod_{j=0}^{\lfloor N/L\rfloor} S_L(b^{jL}x) \d x.
	\]
	Combining this with our previous bounds, we get
	\[
	\liminf_{N\to\infty} \frac{-\log\Big(\int_0^1 \prod_{j=0}^{N-1} g(b^j x) \d x \Big)}{\log b^N}
	\ge
	\frac{-\log\Big(\max_{x}b^{-L}\sum_{i=0}^{b^L-1} S_L(x+i/b^L)\Big)}{\log b^L}
	\]
	and 
	\begin{align*}
	\limsup_{N\to\infty} &\frac{-\log\Big(\int_0^1 \prod_{j=0}^{N-1} g(b^j x) \d x\Big)}{\log b^N}\\
	&\le
	\frac{-\log\Big(\min_{x}b^{-L}\sum_{i=0}^{b^L-1} S_L(x+i/b^L)\Big)}{\log b^L}
	\end{align*}
	Proposition \ref{pr:convergence} implies that the lower and upper bounds converge to
	the same limit as $L\to \infty$, and this completes the proof.
\end{proof}

\begin{lem}\label{lm:cutoff}
	Let $N \in \bN$ and $x\in[-b^N/2, b^N/2]$.
	Then
	\[
	\Big|\log| \hat \nu(x)|-\log\Big(\prod_{j=1}^N g(b^{-j} x)\Big)\Big|\le C_0.
	\]
\end{lem}

\begin{proof}
	Note that, for all $j\ge N+1$, we have $b^{-j}x\in[-1/(2b),1/(2b)]$.
	By Item (3) of Lemma \ref{lm:f}, this gives us
	\begin{align*}
	\prod_{j=N+1}^{\infty} g(b^{-j} x) 
	\ge \exp(-C_0\sum_{j=N+1}^\infty b^{-2j}x^2)
	&\ge \exp(-C_0\sum_{j=1}^\infty b^{-2j}/4)\\
	&\ge\exp(-C_0).
	\end{align*}
	Now the claim follows from \eqref{AbsExact}.
\end{proof}

\begin{proof}[Proof of Proposition \ref{pr:limit-exists}]
	We first prove
	\begin{align*}
	\lim_{N\to\infty} &\frac{-\log\Big(\int_0^1 \prod_{j=0}^{N-1} g(b^j x) \d x\Big)}{\log b^N}\\
	&=\lim_{N\to\infty} \frac{-\log\Big(b^{-N}\sum_{a=0}^{b^N-1} \prod_{j=0}^{N-1} g(b^j (a/b^{N})) \Big)}{\log b^N}.
	\end{align*}	
	To this end, we note that
	\[
	\int_0^1 \prod_{j=0}^{N-1} g(b^j x) \d x=
	\int_0^1 b^{-N}\sum_{a=0}^{b^N-1} \prod_{j=0}^{N-1} g(b^j (x+a/b^{N})) \d x.
	\]
	By the intermediate value theorem, there is some $x_0\in \R/\Z$ such that
	\[
	\int_0^1 \prod_{j=0}^{N-1} g(b^j x) \d x
	=b^{-N}\sum_{a=0}^{b^N-1} \prod_{j=0}^{N-1} g(b^j (x_0+a/b^{N})).
	\]
	We apply Proposition \ref{pr:convergence} with $L=N$ and both with $x_1=0$, $x_2=x_0$
	and $x_1=x_0$, $x_2=0$.
	We get that
	\begin{align*}
	&\frac{\Big|\log\Big(\int_0^1 \prod_{j=0}^{N-1} g(b^j x) \d x\Big)-
	\log\Big(b^{-N}\sum_{a=0}^{b^N-1} \prod_{j=0}^{N-1} g(b^j (a/b^{N}))\Big)\Big|}{N}\\
	&\qquad\qquad\qquad\qquad\qquad\qquad< C/ N^{1/4}
	\end{align*}
	for some constant $C$.
	This proves the claim.
	
We observe that
\[
\sum_{a=0}^{b^N-1} \prod_{j=0}^{N-1} g(b^j (a/b^{N}))=
\sum_{a=0}^{b^N-1} \prod_{j=1}^{N} g(b^{-j} a)=
\sum_{a=-\lfloor b^N/2\rfloor}^{\lceil b^N/2\rceil-1} \prod_{j=1}^{N} g(b^{-j} a),
\]
where, for the second equality, we used that the function $\prod_{j=1}^N g(b^{-j} x)$
is $b^N$-periodic.
By Lemma \ref{lm:cutoff}, we now have
\[
\Bigg|\log\Bigg(
\sum_{a=0}^{b^N-1} \prod_{j=0}^{N-1} 
g(b^j (a/b^{N}))\Bigg)
-\log\Bigg(
\sum_{a=-\lfloor b^N/2\rfloor}^{\lceil b^N/2\rceil-1}|\hat\nu(a)|
\Bigg)\Bigg|
\le C_0.
\]
Using this and $|\hat\nu(-x)|=|\hat \nu(x)|$, we get
\[
\lim_{N\to\infty} \frac{-\log\Big(\sum_{a=0}^{b^N-1} \prod_{j=0}^{N-1} g(b^j (a/b^{N})) \Big)}
{\log b^N}
=\lim_{N\to\infty}
\frac{-\log\Big(\sum_{a=0}^{\lceil b^N/2\rceil-1}|\hat \nu(a)|\Big)}
{\log \lceil b^N/2\rceil}.
\]
	
For any $Q$, there is $N$ such that $ \lceil b^N/2\rceil\le Q\le
\lceil b^{N+1}/2\rceil$ and
\[
\sum_{a=0}^{\lceil b^N/2\rceil-1}|\hat \nu(a)|
\le \sum_{a=0}^{Q-1}|\hat \nu(a)|
\le\sum_{a=0}^{\lceil b^{N+1}/2\rceil-1}|\hat \nu(a)|.
\] 
	This gives
	\[
	\lim_{N\to\infty}
	\frac{-\log\Big(\lceil b^N/2\rceil^{-1}
		\sum_{a=0}^{\lceil b^N/2\rceil-1}|\hat \nu(a)|\Big)}
	{\log \lceil b^N/2\rceil}
	=\lim_{Q\to\infty}
	\frac{-\log\Big(Q^{-1}\sum_{a=0}^{Q-1}|\hat\nu(a)|\Big)}
	{\log Q}.
	\]
	The existence of this limit implies that it must equal $\hat\kappa_1(\nu)$, which
	completes the proof.
\end{proof}

\section{Analytic bounds}\label{sc:analytic}

The purpose of this section is to prove Proposition \ref{pr:kap1>12} for $b\ge 111$
and Proposition \ref{pr:kap1kap>12} for $b\ge 112$.
We also prove Theorem \ref{thm: l1 bound for special arrangement}.

To these ends, we prove a simple analytic bound for some exponential sums in the
next lemma.
The results will follow from this using Theorem \ref{th:bounds}.

\begin{lem}\label{lm:expsum}
Let $3\le l \le b$ and $a$ be integers, and let $d\in\Z_{\ge 1}$ be coprime to $b$.
Then
\[
\sum_{j=0}^{b-1}\Big|\sum_{k=0}^{l-1}e((x+j/b)(dk+a))\Big|\le b(1+\log(2l)) + 3l +2
\]
for all $x\in\R$.
\end{lem}

\begin{proof}
We assume, as we may, that $a=0$, for a factor of $e((x+j/b)a)$ can be pulled out
of the inner sum, which does not affect its absolute value. Note that
\[
\sum_{j=0}^{b-1}\Big|
\sum_{k=0}^{l-1}
e((x+j/b)dk)\Big|
=\sum_{j=0}^{b-1}
f(xd+jd/b)
=\sum_{j=0}^{b-1}
f(xd+j/b),
\]
where
\[
f(t)=\Big|\sum_{k=0}^{l-1}
e(tk)\Big|=\Big|
\frac{e(tl)-1}{e(t)-1}\Big|.
\]
Here we have used that $d$ is coprime to $b$.

We record a few estimates for the function $f$.
First, for all $t$ we have $0 \le f(t) \le l$.
Second, using the estimates
\[
2\ge |e(y)-1|=2\sin(\pi\|y\|)\ge 4\|y\|,
\]
where $\|y\|$ is the distance of $y$ to the nearest integer, we have
\[
f(t) \le \frac{1}{2\|t\|}
\]
for all $t$.
Third, we have $|e(t)-1|\ge 1$ for $t\in[1/6,5/6]$, which gives us
\[
f(t) \le 2
\]
for all such $t$.

In what follows, we assume as we may that $xd\in[0,1/b)$.
Note that
\[
\sum_{j=0}^{b-1}
f(xd+j/b)
\one(xd+j/b\in[0,1/2l)\cup[1-1/2l,1))
\le l(b/l+1).
\]
In addition,
\begin{align*}
\sum_{j=0}^{b-1}
f(xd+j/b)
\one(xd+j/b\in[1/2l,1/6))
&\le l+b\int_{1/2l}^{1/6}\frac{1}{2t}dt\\
&= l + \frac{b}{2}(\log(2l)-\log 6),
\end{align*}
and similarly
\[
\sum_{j=0}^{b-1}
f(xd+j/b)
\one(xd+j/b\in[5/6,1-1/2l))
\le l + \frac{b}{2}(\log(2l)-\log 6).
\]
Finally,
\[
\sum_{j=0}^{b-1}
f(xd+j/b)
\one(xd+j/b\in[1/6,5/6))
\le 2(2b/3+1).
\]

Putting together our estimates and using  that $4/3<\log 6$, we get
\[
\sum_{j=0}^{b-1}
f(xd+j/b)
\le b(1+\log(2l)) + 3l +2,
\]
as required.
\end{proof}

\begin{lem}
Let $b \in\Z_{\ge 3}$, and suppose $D \subset
\{0,1,\ldots,b-1\}$ 
with
$\#D=b-1$.
Then
\[
\hat\kappa_1(\nu_{b,D})\ge\frac{\log (b-1)}{\log b}-\frac{\log(5+\log(2b)+2/b)}{\log b}.
\]
\end{lem}

A simple numerical calculation shows that the lemma implies Proposition \ref{pr:kap1>12}
for $b\ge 111$ and Proposition \ref{pr:kap1kap>12} for $b\ge112$.

\begin{proof}
Taking $L=1$ in Theorem \ref{th:bounds}, we have
\begin{equation}\label{eq:recall-lower}
\hat\kappa_1(\nu)\ge \frac{-\log\max_xb^{-1}\sum_{i=0}^{b-1}g(x+i/b)}{\log b}.
\end{equation}
We write
\[
\sum_{i=0}^{b-1}(b-1)g(x+i/b)=\sum_{i=0}^{b-1}\Big|\sum_{k\in D}e((x+i/b)k)\Big|
\le\sum_{i=0}^{b-1}\Big|\sum_{k=0}^{b-1}e((x+i/b)k)\Big| + b.
\]
Using Lemma \ref{lm:expsum} with $a=0$, $d=1$ and $l=b$, we get
\[
\sum_{i=0}^{b-1}(b-1)g(x+i/b)\le b(5+\log(2b))+2.
\]
We plug this into \eqref{eq:recall-lower} and get the claim of the lemma.
\end{proof}

\begin{proof}[Proof of Theorem \ref{thm: l1 bound for special arrangement}]
The upper bound for $\hat \kappa_1(\nu)$ is clear because 
\[
\hat\kappa_1(\nu)\leq \hat\kappa_2(\nu) = 
\dimH(\nu).
\]

For the lower bound, we use Lemma \ref{lm:expsum} and get
\[
\sum_{i=0}^{b-1}lg(x+i/b)\le b(1+\log(2l))+3l+2.
\]
We plug this into \eqref{eq:recall-lower}, and get
\[
\hat\kappa_1(\nu)\ge\frac{\log l}{\log b}- \frac{\log(1+\log(2l)+(3l+2)/b)}{\log b}.
\]
This proves the claim, since $\dimH\nu=\log l/\log b$ and $3l+2<3b$.
\end{proof}

\section{Numerical estimates}\label{sc:numeric}

The purpose of this section is to prove Propositions \ref{pr:kap1>12} and
\ref{pr:kap1kap>12} in the remaining cases of $b\le 111$.
To this end, we estimate $\hat\kap_1(\nu)$ numerically using Theorem \ref{th:bounds}. We compute
\begin{equation}\label{eq:FL}
F_L(x):=
\sum_{i=0}^{b-1}S_L(x+i/b^L)=\sum_{i=0}^{b-1}\prod_{j=0}^{L-1}g(b^j(x+i/b^L))
\end{equation}
along a sufficiently dense arithmetic progression, and bound the Lipschitz constant to understand the behaviour between these points. The inequality
\[
|e(\tet) - 1| \le 2 \pi |\tet|
\]
assures us that $\Lip(g) \le 2\pi(b-1)$.
As $\|g\|_\infty \le 1$, it then follows from the telescoping identity
\[
a_1 \cdots a_s - b_1 \cdots b_s = \sum_{i \le s} (a_i - b_i)
\prod_{j < i} a_j \prod_{j > i} b_j
\]
that
\begin{equation} \label{Lipschitz}
	\Lip(F_L) \le b^L\Lip(S_L)
	\le b^L  \sum_{j=0}^{L-1} 2\pi(b-1)\cdot b^j = 2\pi b^L (b^L - 1).
\end{equation}

For $\delta>0$, denote by $A_{\delta,L}$ the arithmetic progression
$k\delta$ for $k=0,\ldots,\lceil (2b^L\delta)^{-1}\rceil$.
If we show that
\begin{equation}\label{eq:numeric}
\max_{x\in A_{\delta,L}} F_L(x) < b^{(1-\tau)L} - \delta b^L (b^L - 1)\pi,
\end{equation}
For some $\tau$, $\delta$ and $L$, then Theorem \ref{th:bounds} and the above Lipschitz estimate imply that 
$\hat \kap_1(\nu)
\ge \tau$.

\bigskip

Using the software \emph{SageMath}~\cite{Sage}, we verified \eqref{eq:numeric}
for the missing-digit measures $\nu=\nu_{b,\{0,\ldots,b-1\}\backslash\{a\}}$
for $b=4$ and $a=0$, and for 
\[
b\in\{5,6\},
\qquad
a
\in
\{0,\ldots,\lfloor(b-1)/2\rfloor\}
\]
with $\tau=1/2$, $L=2$ and $\delta=10^{-5}$, except that
for $b=5$, $a=1$, we chose $L=4$ and $\delta=5\cdot10^{-7}$.
We note that the missing-digit measure
$\nu_{b,\{0,\ldots,b-1\}\backslash\{b-1-a\}}$
is the image of 
$\nu_{b,\{0,\ldots,b-1\}\backslash\{a\}}$
under the map 
$x \mapsto 1-x$, 
which does not change Fourier $\ell^1$
dimension.
For this reason, it is enough to do the calculations for 
$a\in\{0,\ldots,\lfloor(b-1)/2\rfloor\}$.
This proves Proposition \ref{pr:kap1>12} for $b\le 6$.

\begin{remark}
We note that $\hat\kap_1(\nu)<1/2$ for the remaining choices of parameters,
namely
\begin{equation} \label{exceptional}
	(b,a) \in \{ (3,0),(3,1),(3,2), (4,1), (4,2)\}.
\end{equation} 
To show this, it is enough to verify 
\begin{equation}\label{eq:numeric2}
	\min_{x\in A_{\delta,L}} F_L(x) > b^{(1-\tau)L} + \delta\cdot 2b^L (b^L - 1)\pi,
\end{equation}
for some choices of $L$ and $\delta$ and $\tau=1/2$, due to the upper bound
for $\hat\kappa_1(\nu)$ in Theorem \ref{th:bounds}.
We carried this out with $L=2$ and $\del = (10)^{-4}$ in all cases given by \eqref{exceptional}.
\end{remark}

Again using the software \emph{SageMath} \cite{Sage}, we verified \eqref{eq:numeric}
with 
\[
\tau = \log b/(2\log(b-1))
\]
in the cases $b\in\{7,8\}$,
$a\in\{0,\ldots,\lfloor(b-1)/2\rfloor\}$ with $L=2$,
$\delta=10^{-5}$ and in the cases $b\in\{9,\ldots,111\}$,
$a\in\{0,\ldots,\lfloor(b-1)/2\rfloor\}$ with $L=1$,
$\delta=10^{-4}$.
This proves both Propositions \ref{pr:kap1>12} and \ref{pr:kap1kap>12}
in the cases $b=7,\ldots,111$.
The proofs of both propositions are now complete.

In the appendix, we provide all of the SageMath code used to carry out the calculations in this section.

\appendix

\section{Code}

Here we provide the SageMath code we used to carry out
the numerical calculations in the paper.
The program uses interval arithmetic, which means that instead
of storing a single floating point number for a variable,
it stores two, a rigorous upper and a rigorous lower bound.

\medskip

\begin{Verbatim}[fontsize=\tiny]
sage: reset()
sage: sage.rings.real_mpfi.printing_style = 'brackets' 
                    # Prints rigorous lower and upper bounds

sage: pi_=RIF(pi)   # Converts pi to an interval
sage: I_= CIF(I)    # Converts I to a complex box


sage: def e_ (x): # Computes the function e(x) defnied in (2.1)
....: 	return exp(RIF(2)*I_*pi_*x)

sage: def g (a,b,x): # Computes the function g(x) defined
                     # in (4.1) for the natural measure on
                     # K_b,D with D={0,...,b-1}\{a}

                     # First, we compute the sum e_(jx) for j=0..b-1

....: 	if (abs(e_(x)-RIF(1))+RIF(-10^(-10),10^(-10))).contains_zero() :
                     # If x is close to an integer, we just compute the sum

....: 		Su=RIF(0)
....:		for j in [0..b-1]:
....: 			Su=Su+e_(RIF(j)*x)
....: 	else:        # Otherwise, we use the closed formula

....: 		Su=(e_(RIF(b)*x)-RIF(1))/(e_(x)-RIF(1))
....: 	return abs((Su-e_(RIF(a)*x))/RIF(b-1))

sage: def S (a,b,L,x): # Computes the function S_L(x)
                       # defined in (4.3)
....:	Pr=RIF(1)
....: 	for j in [0..L-1]:
....: 		Pr=Pr*g(a,b,RIF(b^j)*x)
....:	return Pr

sage: def F (a,b,L,x): # Computes the function F_L(x)
                       # defined in (6.1)
....: 	Su = RIF(0)
....: 	for j in [0..b^L-1] :
....: 		Su = Su+S(a,b,L,x+RIF(j/b^L))
....: 	return Su

sage: def TEST_6_3 (a,b,L,delta,tau): # Returns a positive value
                                      # if (6.3) holds.
....: 	max_ =RIF(0)
....: 	limit=ceil((1/(RIF(2*b^L)*delta)).upper())
....: 	for k in [0..limit]:
....: 		FF=F(a,b,L,RIF(k)*delta)
....: 		if max_ <  FF:
....: 			max_ = FF
....:	return RIF(b)^((RIF(1)-tau)*RIF(L))-delta*RIF(b^L*(b^L-1))*pi_-max_

sage: def TEST_6_5 (a,b,L,delta,tau): # Returns a negative value
                                      # if (6.5) holds.
....: 	min_ =RIF(b^L)
....: 	limit=ceil((1/(RIF(2*b^L)*delta)).upper())
....: 	for k in [0..limit]:
....: 		FF=F(a,b,L,RIF(k)*delta)
....: 		if min_ <  FF:
....:			min_ = FF
....: 	return RIF(b)^((RIF(1)-tau)*RIF(L))-delta*RIF(b^L*(b^L-1))*pi_-min_


# Proving Proposition 2.4 for b<=6

sage: L=2
sage: delta=RIF(10^(-5))
sage: tau=RIF(1/2)

sage: for [b,a] in [[4,0], [5,0], [5,2], [6,0], [6,1], [6,2]] :
....: print("b=", b,"a=", a,":",TEST_6_3(a,b,L,delta,tau))

sage: L=4
sage: delta=RIF(5*10^(-7))
sage: tau=RIF(1/2)
sage: b=5
sage: a=1

sage: print("b=", b,"a=", a,":",TEST_6_3(a,b,L,delta,tau))

# Proving (6.5) for the parameters in (6.4)

sage: L=2
sage: delta=RIF(10^(-4))
sage: tau=RIF(1/2)

sage: for [b,a] in [[3,0],[3,1],[4,1]] :
....: 	print("b=", b,"a=", a,":",TEST_6_5(a,b,L,delta,tau))

sage: def Tau (b): # Computes the value of tau needed to prove Proposition 2.5
....:	return log(RIF(b))/(RIF(2)*log(RIF(b-1))) 

# Proving Proposition 2.5 and the remaining cases of Proposition 2.4

sage: L=2
sage: delta=RIF(10^(-5))

sage: for b in [7,8] :
....: 	for a in [0..floor((b-1)/2)] : 
....:		print("b=", b,"a=", a,":",TEST_6_3(a,b,L,delta,Tau(b)))

sage: L=1
sage: delta=RIF(10^(-4))

sage: for b in [9..111] :
....: 	for a in [0..floor((b-1)/2)] : 
....: 		print("b=", b,"a=", a,":",TEST_6_3(a,b,L,delta,Tau(b)))
	
\end{Verbatim}

\providecommand{\bysame}{\leavevmode\hbox to3em{\hrulefill}\thinspace}

\end{document}